\documentclass[12pt,a4paper]{article}
\usepackage{url}
\usepackage[utf8]{inputenc}
\usepackage{tikz}
\usepackage{cancel}
\usepackage{color}
\usepackage[normalem]{ulem}
\usepackage[a4paper,left=3.5cm,right=3.5cm,top=3cm,bottom=3.5cm]{geometry}
\usepackage[english]{babel}
\usepackage{color}
\usepackage[english]{babel}
\usepackage{color}
\usepackage{graphicx}
\usepackage{amsmath,amssymb,amsthm}
\usepackage[T1]{fontenc}
\usepackage[utf8]{inputenc}
\usepackage{authblk}
\usepackage[english]{babel}
\usepackage{amsmath,amssymb,amsthm}
\usepackage{color}
\usepackage{graphics}
\usepackage{amsmath,amssymb,amsthm}
\usepackage{lipsum}
\usepackage{graphicx}
\usepackage{amsmath,amssymb,bm}
\usepackage{url}

\newtheorem{theorem}{Theorem}
\newtheorem{proposition}[theorem]{Proposition}
\newtheorem{corollary}[theorem]{Corollary}
\newtheorem{lemma}[theorem]{Lemma}
\theoremstyle{definition}
\newtheorem{definition}[theorem]{Definition}
\newtheorem{remark}[theorem]{Remark}
\newtheorem{example}{Example}

\title{Unified Common-Root and Interpolation Bounds Based on Leading Monomial Data}

\author{Olav Geil\footnote{olav@math.aau.dk}
}
\affil{Department of Mathematical Sciences\\ Aalborg University}

\begin{document}
\maketitle

\begin{abstract}
For general fields the footprint bound from Gröbner basis theory estimates the number of common affine roots of any set of multivariate polynomials using information on their leading monomials. In this paper we develop an interpolation bound with a similar flavor extending a previously known result for only a single polynomial to any prescribed number of polynomials. Surprisingly, our interpolation theorem and the footprint bound can be shown to be two sides of the same coin, solving similar problems, but for dual spaces. As discussed the footprint bound compares well with the improved Alon-Füredi bound and for finite fields the presented interpolation theorem is sharp. Our work can be viewed as a comment to a question raised by Tao in~\cite{tao2014algebraic}.\\

\noindent {\bf{Keywords:}} {\it{Affine roots, footprint bound, generalized Alon-Füredi bound, interpolation, multivariate polynomials, polynomial method}}\\

\noindent {\bf{MSC}}: 05E40; 12E05
\end{abstract}

\section{Introduction}\label{secintro}
As is well-known for a univariate non-zero polynomial over an arbitrary field the number of roots is at most equal to its degree, and if the field is algebraically closed and if roots are counted with multiplicity then equality holds. This result is in close correspondence with the fact that the interpolation problem of finding a polynomial of lowest possible degree evaluating to zero on a finite point-set boils down to calculating the product of linear terms each having one of the desired elements as root, and by repeating terms one can also deal with multiplicity. The situation is very much in contrast to that of multivariate polynomials. Firstly, a multivariate polynomial ring is not a principal ideal domain, and therefore the interpolation problem can not  in general be about multiplying linear factors. Secondly, a non-constant multivariate polynomial often has infinitely many roots and always does if the field is algebraically closed, and therefore if the aim is to consider a case of  finitely many roots then one either needs to restrict to roots from a prescribed finite point-set or one must study mutual roots of a set of polynomials which together define a zero-dimensional ideal. 

The purpose of our work is to consider a natural enhancement of the results from univariate polynomials into the case of common affine roots (counted without multiplicity) of general multivariate polynomials and by doing so to establish a connection between the following two problems:
\begin{itemize}
\item[(I)] How many common affine roots can a finite set of polynomials have given some  measure on their sizes?
\item[(II)] Given some measure on the size of polynomials consider the class of all those of size within some given interval. What is the largest number  $a$ such that for any affine point-set $A$ of cardinality equal to $a$ there exists $k$ linearly independent polynomials within the class evaluating simultaneously to zero on $A$, where $k$ is some adequate integer? 
\end{itemize} 

One often used measure on the size of a polynomial is its total degree in which case one may for instance in (I) and (II) consider the class of all polynomials of a given prescribed total degree. A much more fine grained measure, however, would be that of the leading monomial according to some fixed monomial ordering. The contribution of the present paper is to demonstrate that by applying the latter mentioned measure one is able to obtain for multivariate polynomials answers to (l) and (II) which are as strongly related to each other as are the answers in the univariate case. The answer to (I) in this context is the well-known footprint bound, but our answer to (II) is novel. Recall that for univariate polynomials there is only one choice of monomial ordering and the leading monomial of a polynomial $F(X)$ of degree $d\geq 0$ equals $X^d$, hence the two concepts are basically the same. In contrast, for multivariate polynomials there are infinitely many monomial orderings and for many of these the total degree of the leading monomial does not even need to be equal to the total degree of the polynomial under consideration. Considering leading monomial(s) is not only as natural a choice for our study as considering total degree(s); as shall be demonstrated it is the key to achieve the above mentioned enhancement and connection. Even more importantly, using the leading monomial rather than the total degree provides much more fine-grained information which is  sharp in several cases.

In his exposition~\cite{tao2014algebraic} Tao treats problems (I) and (II) from the above paragraph applying the total degree as the measure of the size of polynomials and mostly concentrating on single polynomials rather than arbitrary sets of polynomials. For the case of affine roots he discusses on the one side Alon's famous combinatorial Nullstellensatz~\cite{alon1999c,kos2012a}, the Alon-Füredi bound~\cite{alon1993covering}, the Schwartz-Zippel bound~\cite{1979probabilistic,schwartz1980fast} (as well as Dvir's generalization taking multiplicity into account~\cite{dvir2009s,dvir2020i,dvir2013extensions}) and on the other side the below standard result~\cite{tao2014algebraic}[Lem.\ 1] on interpolation; with the aim of solving problems in extremal combinatorics: 
\begin{proposition}~\label{lem1}
Let ${\mathbb{F}}$ be a field, let $m \geq 1$ be an integer, and $d \geq 0$. If $A \subseteq {\mathbb{F}}^m$ has cardinality less than ${d+m}\choose{m}$, then there exists a non-zero polynomial $F \in {\mathbb{F}}[X_1, \ldots , X_n]$ of degree at most $d$ having $A$ as roots.
\end{proposition}
 At pages 27-28 Tao comments on the relationship between common-root and interpolation bounds. He writes: \\
 
 \hspace{0.4cm} \begin{minipage}{13cm} 
 {\textit{Unfortunately, the two methods cannot currently be easily combined, because the polynomials produced by interpolation methods are not explicit enough that individual coefficients can be easily computed, but it is conceivable that some useful unification of the two methods could appear in the future.}} \end{minipage}  \\
 
The theory of commutative algebra and algebraic geometry is very rich and indeed Tao in addition to the above results on affine roots treats a cornucopia of deep results on projective roots counted with multiplicity such as Bezout's Theorem, Stepanov's Method, the Hasse Bound and many more~\cite{baker1968linear,bombieri2006counting,hasse1936theorie,thas1995projective,lang1954number,stepanov1999counting,weil1949numbers}, as well as provide interpolation counterparts in more cases, however, the latter only for single polynomials. In this setting Bezout's theorem for the number of common projective roots counted with intersection multiplicity provides a complete answer to an adapted version of (I), but the problem of estimating affine roots which is part of the exposition in the present paper has a different nature making other methods relevant. 

In the present paper we give refined and enhanced answers to both (I) and (II) using the leading monomials of a set of polynomials as a measure on their sizes. Knowing the leading monomial corresponds to having important information regarding the most crucial coefficient and thereby our insight addresses the remark by Tao. More importantly, we demonstrate a very strong relationship between our answers to (I) and (II) which is then our proposal for the requested useful unification, except we do not treat multiplicity.  

 Our contribution is to address the problem of common affine roots of general polynomials where the only assumption is that no two of them have the same leading monomial which is no restriction as we can always apply the division algorithm~\cite{clo}[Sec,\ 2, Par.\ 3]] without effecting the set of common roots. Similarly, we consider any finite affine point-sets without any restriction on their structure. Therefore, our answer to (II) should be compared to Proposition~\ref{lem1} and our answer to (I) should be compared to the Schwartz-Zippel bound, the Alon-Füredi bound and the generalized Alon-Füredi bound~\cite{bishnoi2018zeros}. In both cases the material of the present paper can be viewed as improvements and of course also as enhancements as we in contrast to the mentioned bounds do not restrict to single polynomials but rather consider arbitrary finite sets of linearly independent polynomials. 
 

Our answer to (I) is basically to recall some instances of the footprint bound~\cite{clo2,onorin,geilhoeholdt}, but we choose to present in this paper as an answer to (I) also an enhanced version of the Feng-Rao bound for primary linear codes~\cite{AG,GeilEvaluationCodes,agismm} demonstrating that this bound works over arbitrary fields rather than only over finite fields. The latter framework is more general as our footprint bound supported answer to (I) can be derived as an easy corollary to the Feng-Rao bound for primary codes. More importantly, the natural framework for us to prove our answer to (II) is to combine some basic linear algebra results stated by Forney in~\cite{forney94} with an enhanced version of the Feng-Rao bound for dual linear  codes~\cite{FR24,FR1,FR2,handbook,geithom} (although we could have also chosen to replace the latter with an enhanced version of the order bound~\cite{handbook} from order domain theory which would represent an intermediate level between the polynomial ring description and the pure linear algebra description). As shall become clear from our analysis, in a certain sense problems (I) and (II) are basically the same, but stated for dual spaces. This duality is also what allows us to prove that the answer to (II) is sharp in the finite field case. We stress that our description at linear code level should not be considered as a detour for us to treat the case of general sets of multivariate polynomials. In our view the description at linear code level is the most fundamental description with the potential of supporting future investigations beyond the level of general sets of multivariate polynomials and general affine point-sets over fields.

The paper is organized as follows. We start in Section~\ref{secmain} by presenting in Theorem~\ref{theint} and Theorem~\ref{theintfinite} our main finding regarding (II) and by recalling various versions of the footprint bound which address (I). In that section we also in Remark~\ref{remconnect} provide a description of the relationship between these results. Furthermore, we prove the superiority of the footprint bound over the Alon-Füredi bound. To prepare for the proof of the first mentioned result we then in Section~\ref{secclean} treat related problems in the language of error-correcting codes over arbitrary fields. In Section~\ref{seccorollary1} as a
corollary we then obtain the proof of our answer to (II) 
which also demonstrates the strong relationship between (I) and (II), and which implies that our answer to (II) is sharp in the case of finite fields. Section~\ref{secconclusion} contains concluding remarks.

\section{Main results on multivariate polynomials}\label{secmain}
The main ingredient of the interpolation theorem is the following concept.

\begin{definition}
For multivariate monomials $N_1, \ldots , N_s$ write
\begin{eqnarray}
&&\mu(N_1, \ldots , N_s) :=\nonumber \\
&&{\mbox{ \ \ \ \ \ }}\# \{ M \mid M {\mbox{ is a monomial which divides some }} N_i, i \in \{1, \ldots , s\} \}. \nonumber
\end{eqnarray}
\end{definition}

For readability reasons we state the interpolation theorem as two separate theorems.
\begin{theorem}\label{theint}
Let ${\mathbb{F}}$ be an infinite field. Given a monomial ordering $\prec$ on the set of monomials in $m$ variables let $M_1 \prec \cdots \prec M_t$ be strictly consecutive monomials, i.e.\ for $i=1, \ldots , t-1$ there does not exist any monomial $L$ with $M_i \prec L \prec M_{i+1}$. 
Consider $1 \leq k \leq t$ then for any $A \subseteq {\mathbb{F}}^m$ with 
\begin{equation}
\# A  < \min \{ \mu (M_{i_1}, \ldots , M_{i_{t-k+1}}) \mid 1 \leq i_1 < \cdots < i_{t-k+1} \leq t\}\label{eqsnabeleinz}
\end{equation}
there exists $k$ polynomials $F_1, \ldots , F_k$ with pairwise different leading monomials all belonging to $\{M_1, \ldots , M_t\}$ in such a way that every element of $A$ is a common root.
\end{theorem}
\begin{proof}
For $k=1$ the proof of Proposition~\ref{lem1} as presented in~\cite{tao2014algebraic} trivially applies in the following way. Consider the linear map 
$${\mbox{ev}}: {\mbox{Span}}_ {\mathbb{F}} \{ M \mid M {\mbox{ divides some }} M_i, i \in \{ 1, \ldots , t\} \} \rightarrow  {\mathbb{F}}^{\#A}$$ 
given by ${\mbox{ev}}(F)=(F(P_1), \ldots , F(P_{\# A\ }) )$ where $A=\{P_1, \ldots , P_{\# A}\}$. As the preimage is larger than the image there must exist a non-zero $F$ in the preimage evaluating to $\vec{0}$. Now ${\mbox{lm}}(F)$ divides some $M_i$ and we write $K=\dfrac{M_i}{{\mbox{lm}(F)}}$. But then ${\mbox{lm}}(K F)=M_i$ and $KF$ of course evaluates to $\vec{0}$. 
The proof for larger values of $k$ can be found in Section~\ref{seccorollary1}.
\end{proof}

Following the tradition in the literature (e.g.\ the Schwartz-Zippel bound) we shall give a special treatment to the situation where we restrict to point sets being subsets of some fixed finite Cartesian product. 
First, observe that if ${\mathbb{F}}$ is a finite field with $q$ elements, in which case we write ${\mathbb{F}}={\mathbb{F}}_q$, and if $q \leq d$ then Proposition~\ref{lem1} does not give us any real information as all points in ${\mathbb{F}}^n$ are then trivially roots of each of the polynomials $X_1^q-X_1, \ldots , X_m^q-X_m$. In larger generality assume we are restricting to point sets $A \subseteq A_1 \times \cdots \times A_m$ where $A_s \subseteq {\mathbb{F}}$ is a fixed finite set with say $a_s=\#A_s$ for $s=1, \ldots , m$. An adapted version of Proposition~\ref{lem1} produces no real information when $\min\{ a_1, \ldots , a_m\} \leq d$ as $a_s\leq d$ implies that $G_s:=\prod_{b \in A_s}(X_s-b)$ has the entire set $A_1 \times \cdots \times A_m$ as roots. The below Theorem~\ref{theintfinite} is a modified version of Theorem~\ref{theint} to also give non-trivial insight in such cases. 

In the remainder of the paper when writing $A_1 \times \cdots \times A_m \subseteq {\mathbb{F}}$ we shall always assume the Cartesian product to be a finite set and we use the notation $a_s$ and $G_s$ in the above meaning.  We need to introduce some additional notation. 

\begin{definition}\label{defthree}
For a finite point set $A_1 \times \cdots \times A_m\subseteq {\mathbb{F}}^m$ define
$$\Box(a_1, \ldots , a_m):=\{X_1^{i_1} \cdots X_m^{i_m} \mid 0 \leq  i_s < a_s, {\mbox{ for }} s=1, \ldots , m\}.$$
\end{definition}

\begin{remark}
Note, that if we write $A_1\times \cdots \times A_m=\{P_1, \ldots , P_{n=a_1 \cdots a_m} \}$ then 
$$\{ (M(P_1), \ldots , M(P_n)) \mid M \in \Box (a_1, \ldots , a_m)\}$$
constitutes a basis for ${\mathbb{F}}^{n}$ as a vector-space over ${\mathbb{F}}$. In particular for any polynomial $G$ with support not being a subset of $\Box(a_1, \ldots , a_m)$ there exists a polynomial $H$ with support being a subset of $\Box(a_1, \ldots , a_m)$ in such a way that $G$ and $H$ evaluates similarly at $A_1 \times \cdots \times A_m$ and such that if $H$ is not the zero-polynomial then ${\mbox{lm}}(H) \preceq {\mbox{lm}}(G)$ holds for any choice of monomial ordering $\prec$. Here, ${\mbox{lm}}()$ denotes the leading monomial. More concretely, $H$ can be obtained from $G$ by dividing the latter with $\{ G_1, \ldots , G_m\}$. Therefore it is no restricting to consider in the case of point sets being subsets of a finite Cartesian set only polynomials with support in $\Box(a_1, \ldots , a_m)$. 
\end{remark}
Given a polynomial $F$ in the following by ${\mbox{Supp}}(F)$ we shall denote the support of it, i.e. ${\mbox{Supp}}(F)$ is the set of monomials that appear with non-zero coefficient in the description of $F$. We note that later in the paper we shall also use the notion ${\mbox{Supp}}()$ in another meaning, when we apply it to vectors and vector-spaces. We are now ready to state the second version of the interpolation theorem.

\begin{theorem}\label{theintfinite}
Let $A_1 \times \cdots \times A_m \subseteq {\mathbb{F}}^m$ be a finite set. Given a monomial ordering $\prec$ on the set of monomials in $m$ variables let $M_1 \prec \cdots \prec M_t$ be strictly consecutive monomials in  $\Box(a_1, \ldots , a_m)$. I.e. $M_1, \ldots , M_t \in \Box(a_1, \ldots , a_m)$ and if for some $i\in \{1, \ldots , t-1\}$ there exists a monomial $L$ with $M_i \prec L \prec M_{i+1}$ then $L \notin \Box(a_1, \ldots , a_m)$ . 
Consider $1 \leq k \leq t$ then for any $A \subseteq A_1 \times \cdots \times A_m$ with 
\begin{equation}
\# A  < \min \{ \mu (M_{i_1}, \ldots , M_{i_{t-k+1}}) \mid 1 \leq i_1 < \cdots < i_{t-k+1} \leq t\}\label{eqsnabeleinzzwei}
\end{equation}
there exists $k$ polynomials $F_1, \ldots , F_k$ with pairwise different leading monomials all belonging to $\{M_1, \ldots , M_t\}$ in such a way that every element of $A$ is a common root. Without loss of generality one may assume ${\mbox{Supp}}(F_i) \subseteq \Box(a_1, \ldots , a_m)$ for $i=1, \ldots , k$. If $A_1 \times \cdots \times A_m={\mathbb{F}}_q^m$ then~(\ref{eqsnabeleinzzwei}) is sharp  meaning that for some $A$ of cardinality one more such polynomials do not exist.
\end{theorem}
\begin{proof}
For $k=1$ the proof of Proposition~\ref{lem1} as presented in~\cite{tao2014algebraic} trivially applies with a few modifications.
The proof for larger values of $k$ can be found in Section~\ref{seccorollary1}, where also the sharpness in case of a finite field is demonstrated.
\end{proof}

\begin{remark}
Theorem~\ref{theint} reduces to the well-known result for univariate polynomials by choosing $\{M_{t=1}=X^d\}$ where $d$ is the degree under consideration. 
\end{remark}
\begin{remark}
Observe, that in the case of an infinite field or in the case of a finite field ${\mathbb{F}}_q$ with $d< q$, respectively, Proposition~\ref{lem1} appears as a corollary to Theorem~\ref{theint} or Theorem~\ref{theintfinite}, respectively, by choosing a graded monomial ordering and by letting $\{M_1, \ldots , M_t\}$ be all monomials of total degree $d$. For Proposition~\ref{lem1} to cover in the case of a finite field ${\mathbb{F}}_q$ general meaningful $d$, i.e.  $1 \leq d \leq m (q-1)$, we request that the support of $F$ is a subset of $\Box(q, \ldots , q)$  and we replace $\# A < {{d+m}\choose{m}}$ with 
\begin{equation}
\# A < \# \{M \in \Box (q, \ldots , q)\mid \deg (M) \leq d \}. \label{eqhoejre}
\end{equation} 
We note that the right-hand side of (\ref{eqhoejre}) equals the dimension of the q-ary Reed-Muller code ${\mathcal{RM}}_q(d,m)$~\cite[Lem.\ 2.4.6]{delsarte1970generalize}).
\end{remark}
\begin{remark}
Theorem~\ref{theintfinite} can be generalized to cover the case of $A \subseteq S$ for any fixed finite set $S \subseteq {\mathbb{F}}^m$ which does not need to be a Cartesian product. This is done by employing a Gröbner basis for the ideal ${\mathcal{I}}(S) $ in the same way as we employed $\{ G_1, \ldots , G_m\}$ which indeed is a Gröbner basis for ${\mathcal{I}}(A_1 \times \cdots \times A_m)$. We leave the details for the reader.
\end{remark}
As a corollary to Theorem~\ref{theint} and Theorem~\ref{theintfinite} we obtain the following result.
\begin{corollary}\label{cornewthing}
Let $s_2 <s_1$, $s_i \in {\mathbb{N}}$, $i=1,2$ be arbitrary if ${\mathbb{F}}$ is infinite, and satisfying $s_1 \leq q-1$ in the case ${\mathbb{F}}={\mathbb{F}}_q$. Write $\tau=s_1-s_2+1$. The number of monomials $X^iY^j$ with $s_2 \leq i+j \leq s_1$ equals $$t=\tau s_2+\frac{\tau(\tau+1)}{2}.$$ For any $A \subseteq {\mathbb{F}}^2$ satisfying the following condition there exist polynomials $F_1, \ldots  ,F_k \in {\mathbb{F}}[X,Y]$ of total degree in the interval $[ s_2, s_1]$ all evaluating to zero on $A$ and with  ${\mbox{lm}}(F_i) \neq {\mbox{lm}}(F_j)$ for $i\neq j$, where the leading monomial is with respect to a fixed graded lexicographic ordering. For $t-k+1 =1, \ldots , \tau s_2$ write $t-k+1=a \tau+b$ where $a \in \{0, \ldots , s_2-1\}$ and $b \in \{1, \ldots , \tau\}$. The condition on $A$ then is
\begin{equation}
\#A <  \left(s_1+1-\frac{a}{2}\right)(a+1)+b-\tau. \label{eqsnabel52} 
\end{equation}
For  $t-k+1=\tau s_2+c$ with $c\in \{1, \ldots , \frac{\tau(\tau+1)}{2}\}$ the condition on $A$ is
\begin{equation}
\# A < \frac{1}{2}(s_1-\tau+1)(s_1+\tau+2)+c. \label{eqsnabel53}
\end{equation}
\end{corollary}
\begin{proof}
The proof uses results on relative generalized Hamming weights of $q$-ary Reed-Muller codes as presented in~\cite{geilmartin2017relative}. Such results are  translated into information on the function $\mu$ which can be applied also if ${\mathbb{F}}$ is not finite. The details are given in Section~\ref{seccorollary1}.
\end{proof}

\begin{example}
For any field ${\mathbb{F}}$ of size at least $4$ (e.g.\ an infinite field) consider the class of polynomials of total degree $3$ in two variables. From Corollary~\ref{cornewthing} we see that for arbitrary $A \subseteq {\mathbb{F}}^2$ with $\# A < 10$ there exists a polynomial in this class evaluating to zero on $A$. If $\# A < 9$ then there exist two such polynomials with different leading monomials with respect to a fixed graded lexicographic ordering. If $\# A < 7$ there exist three such polynomials with pairwise different leading monomials. Finally, if $\# A < 4$ then there exist four such polynomials with pairwise different leading monomials.
\end{example}

We next address (I) by recalling the footprint bound on the number of common affine roots of any given set of polynomials over arbitrary field~\cite[Ch.\ 2, Thm.\ 2.10]{clo2} and \cite{onorin,geilhoeholdt}. To the best of our knowledge this is the most general and sharp bound for arbitrary sets of polynomials and arbitrary affine point-sets (see, however, Remark~\ref{remnewthing} below).
Let $\prec$ be a monomial ordering on the set of monomials in $m$ variables, and consider an ideal $I \subseteq {\mathbb{F}}[X_1, \ldots , X_m]$. The footprint of $I$ with respect to $\prec$ is
\begin{eqnarray}
 && {\mbox{\hspace{-0.5cm}}} \Delta_\prec (I):= \{M {\mbox{ is a monomial }} \mid M \nonumber \\
 && {\mbox{ \ \ \ \ \ \ \ \  \  \ \ \ \ is not divisible by the leading monomial of any }} F \in I \}. \nonumber
 \end{eqnarray}
The crucial observation is that $\{ M+I \mid M \in \Delta_\prec (I) \}$ constitutes a basis for ${\mathbb{F}}[X_1, \ldots , X_m]/I$ as a vector-space over ${\mathbb{F}}$, \cite[Ch.\ 5, Sec.\ 3, Prop.\ 4]{clo}, which is seen by applying the division algorithm.
\begin{theorem}\label{thefoot}
If $\Delta_\prec (I)$ is finite then the size of the variety $V_{\mathbb{F}}(I)$ is at most equal to $\# \Delta_\prec (I)$ with equality if $I$ is a radical ideal and ${\mathbb{F}}$ is algebraically closed. 
\end{theorem}
\begin{proof}
Assume first that $V_{\mathbb{F}}(I)$ is finite and write $V_{\mathbb{F}}(I)=\{ \beta_1, \ldots , \beta_n\}$.  Consider the evaluation map ${\mbox{ev}}: {\mathbb{F}}[X_1, \ldots , X_m]/I \rightarrow {\mathbb{F}}^n$ given by ${\mbox{ev}}(F+I)=(F(\beta_1), \ldots , F(\beta_n))$. By interpolation arguments this map is surjective and the first part of the theorem follows. To see the last part we use the fact~\cite{clo}[Ch.\ 4, Sec.\ 2, Thm.\ 7] that ${\mathcal{I}}(V_{\mathbb{F}}(I))=I$ whenever $I$ is radical and ${\mathbb{F}}$ is algebraically closed. Aiming for a contradiction we finally assume that $V_{\mathbb{F}}(I)$ is infinite and choose $\{ \beta_1, \ldots , \beta_n\} \subseteq V_{\mathbb{F}}(I)$ with $n > \# \Delta_\prec (I)$. But then ${\mbox{ev}}$ cannot be surjective.
\end{proof}
\begin{remark}\label{remnewthing}
If roots are counted with intersection multiplicity then the bound in Theorem~\ref{thefoot} is  sharp~\cite{clo2}. The following Corollary~\ref{corfoot} extends to such a setting, but the subsequent Corollary~\ref{corsamesame} which plays a crucial role in establishing the bridge between root-counting and our interpolation theorem does not generalize in an obvious manner.
\end{remark}
As a corollary to Theorem~\ref{thefoot} we obtain the following result where the assumption on $F_1, \ldots , F_s$ can always be ensured due to the division algorithm.
\begin{corollary}~\label{corfoot}
Consider $I$ such that $\Delta_\prec(I)$ is finite. Let $F_1, \ldots , F_s \in {\mathbb{F}}[X_1, \ldots , X_m]$ with $F_1, \ldots , F_s \in {\mbox{Span}}_{\mathbb{F}} \Delta_\prec (I)$ satisfying that ${\mbox{lm}}(F_1)=M_1, \ldots , {\mbox{lm}}(F_s)=M_s$ are pairwise different. Then the number of common roots from $V_{\mathbb{F}}(I)$ is at most 
\begin{eqnarray}
&&\# \{ M \in \Delta_\prec (I) \mid M {\mbox{ is not divisible by any }} M_i, i=1, \ldots , s\}\nonumber \\
&=&\# \Delta_\prec (I) - \sigma (M_1, \ldots , M_s) \nonumber 
\end{eqnarray}
where
\begin{eqnarray}
&&\sigma (M_1, \ldots , M_s) \nonumber \\
&& {\mbox{ \ \ \ }}:= \# \{ M \in \Delta_\prec (I) \mid M {\mbox{ is divisible by some }} M_i, i\in \{1, \ldots , s \} \}.\nonumber
\end{eqnarray}
\end{corollary}
As a special case of Corollary~\ref{corfoot} we may consider the situation where the affine variety equals a finite Cartesian product point-set $A_1 \times \cdots \times A_m$. As shall become clear in Remark~\ref{remconnect} below a study of this case surprisingly holds the key to establish the connection between our answers to (I) and (II). For finite Cartesian product point-sets the $I$ in Corollary~\ref{corfoot} becomes
\begin{eqnarray}
I= \left\langle G_1, \ldots , G_m\right\rangle \label{eqsnabeldelta}
\end{eqnarray}
from which we easily read of the footprint which is independent on the actual choice of monomial ordering $\prec$
\begin{equation}
\Delta_\prec (I) =\Box(a_1, \ldots , a_m), \label{eqsnabeldiamond}
\end{equation}
see Definition~\ref{defthree}.

\begin{corollary}\label{corsamesame}
Let $A_1 \times \cdots \times A_m\subseteq {\mathbb{F}}^m$ be a finite point-set, and consider an arbitrary monomial ordering $\prec$. Let $I$ be as in (\ref{eqsnabeldelta}) which implies that $\Delta_\prec (I)$ equals (\ref{eqsnabeldiamond}). Consider pairwise different monomials $M_1, \ldots , M_t \in \Delta_\prec (I)$. For general set of $k \leq t$  polynomials  $F_1, \ldots , F_k \in {\mathbb{F}}[X_1, \ldots , X_m]$ with ${\mbox{Supp}}(F_i) \subseteq \Delta_\prec (I)$, $i=1, \ldots , k$ and leading monomials being pairwise different all belonging to $\{M_1, \ldots , M_t\}$ the number of non-common roots of $F_1, \ldots , F_k$ in $A_1 \times \cdots \times A_m$ is at least 
\begin{eqnarray}
&&\min \{ \sigma (M_{i_1}, \ldots , M_{i_k}) \mid 1 \leq i_1 < \cdots <i_k \leq t \} \nonumber \\
&=&\min \{ \mu \left( \dfrac{K}{M_{i_1}}, \ldots , \dfrac{K}{M_{i_k}} \right) \mid  1 \leq i_1 < \cdots <i_k \leq t \} \label{eqother}
\end{eqnarray}
where $K=X_1^{a_1-1} \cdots X_m^{a_m -1}$. This bound is sharp.
\end{corollary}
\begin{proof}
We only need to demonstrate sharpness. To this end write  $A_i=\{b_1^{(i)}, \ldots , b_{a_i}^{(i)}\}$ for $i=1, \ldots , m$. To each monomial $N\in \Delta_\prec (I)$ we associate the polynomial $$H_N:=\prod_{i=1}^m \prod_{j=1}^{\deg_{X_i}N} (X_i-b_j^{(i)})$$
which clearly has leading monomial $N$.
Then for any set of monomials $\{N_1, \ldots , N_s\}\subseteq \Delta_\prec (I)$ the number of non-common roots of $H_{N_1}, \ldots , H_{N_s}$ is precisely $\sigma(N_1, \ldots , N_s)$.
\end{proof}

A special case of Corollary~\ref{corsamesame} is when $\#A_1=\cdots =\#A_m=:a$ and there is only one polynomial $F$, the only information of which we have is its total degree $\deg F=d$. Writing $d=(\sum_{i=1}^v(a-1))+\ell$ where $0 \leq \ell < a-1$ we obtain the following upper bound on the number of roots
\begin{eqnarray}
\min \{\sigma (M) \mid M \in \Box(a, \ldots , a), \deg M =d\}&=&\sigma (X_1^{a-1} \cdots X_v^{a-1}X_{v+1}^\ell) \nonumber  \\
&=&(a-\ell)a^{m-(v+1)} \nonumber
\end{eqnarray} 
which is an incidence of the Alon-Füredi bound~\cite{alon1993covering} that in its most general version allows for $A_1, \ldots , A_m$ to be of different sizes. The general version of the Alon-Füredi bound can be deduced from Corollary~\ref{corsamesame} in a similar manner.

\begin{example}
Let ${\mathbb{F}}={\mathbb{F}}_8$, i.e.\ the field with $8$ elements, and consider $I=\langle X_1^8-X_1, X_2^8-X_2 \rangle \subseteq {\mathbb{F}}_8[X_1,X_2]$. Clearly, $V_{\mathbb{F}_8}(I)= {\mathbb{F}}_8^2$. We have $\Delta_\prec(I)=\{ X_1^{i_1}X_2^{i_2} \mid 0 \leq i_1,  i_2 < 8\}$ regardless of the choice of monomial ordering $\prec$. This set of course is the set $\Box(8, 8)$ from Definition~\ref{defthree}. Assume $M_1=X_1^3X_2$, $M_2=X_1^2X_2^2$, and $M_3=X_1X_2^3$ and that $\prec$ is the degree lexicographic ordering with $X_1 \prec X_2$ which implies that indeed $M_1, M_2, M_3$ are strictly consecutive. We have $\mu (M_1)=\mu(M_3)=8$, $\mu(M_2)=9$, $\mu (M_1,M_2)=\mu (M_2, M_3)=11$, $\mu (M_1, M_3)=12$, and $\mu (M_1, M_2, M_3)=13$. From Theorem~\ref{theint} we conclude that for each $A \subseteq {\mathbb{F}}_8^2$ of size at most $7$ there exist polynomials $F_1, F_2, F_3$, ${\mbox{lm}}(F_1)=M_1$, ${\mbox{lm}}(F_2)=M_2$, ${\mbox{lm}}(F_3)=M_3$ having all elements of $A$ as common roots. We similarly conclude that whenever $A$ is of size at most $10$ then there exist $F_1, F_2$ with ${\mbox{lm}}(F_1) \neq {\mbox{lm}}(F_2)$, ${\mbox{lm}}(F_1), {\mbox{lm}}(F_2) \in \{M_1, M_2, M_3\}$ with all elements of $A$ as common roots. Finally for any $A$ with $\#A <13$ there exists $F$ with ${\mbox{lm}}(F) \in \{M_1, M_2, M_3\}$ such that $F(a)=0$ for all $a \in A$. 
In the other direction $\sigma(M_1)=\sigma(M_3)=35$, $\sigma (M_2)=36$, $\sigma (M_1, M_2)=\sigma (M_2, M_3)=41$, $\sigma (M_1, M_3)=45$ and $\sigma (M_1, M_2, M_3)=46$ which by Corollary~\ref{corsamesame} tell us that any polynomial $F$ with ${\mbox{lm}}(F) \in \{M_1, M_3\}$ has at most $64-35=29$ roots and that if ${\mbox{lm}}(F)=M_2$ then it has at most $64-36=28$ roots. If ${\mbox{lm}}(F_1)=M_1$ and ${\mbox{lm}}(F_2)=M_2$ or if ${\mbox{lm}}(F_1)=M_2$ and ${\mbox{lm}}(F_2)=M_3$ then $F_1$ and $F_2$ have at most $64-41=23$. If ${\mbox{lm}}(F_1)=M_1$ and ${\mbox{lm}}(F_2)=M_3$ the number of common roots is at most $64-45=19$. Finally $F_1, F_2, F_3$ with ${\mbox{lm}}(F_1)=M_1$, ${\mbox{lm}}(F_2)=M_2$, and ${\mbox{lm}}(F_3)=M_3$ have at most $64-46=18$ roots in common.
\end{example}

The similarity in expression~(\ref{eqsnabeleinz}), (\ref{eqsnabeleinzzwei}) and expression~(\ref{eqother}) suggests that there is a strong relationship between on the one side Theorem~\ref{theint} and Theorem~\ref{theintfinite} and on the other side Corollary~\ref{corsamesame} which we now explain. Note, however, that there in contrast to Theorem~\ref{theint} and Theorem~\ref{theintfinite} is in Corollary~\ref{corsamesame} no requirement that the monomials $M_1, \ldots , M_s$ should be strictly consecutive. 

\begin{remark}\label{remconnect}
Let $M_1, \ldots , M_t$ be as in Theorem~\ref{theint} or Theorem~\ref{theintfinite}, respectively. In the case of the latter mentioned theorem we are already given finite sets $A_1, \ldots , A_m$. In the case of Theorem~\ref{theint} we now introduce similar sets by choosing finite $A_1 , \ldots , A_m \subseteq {\mathbb{F}}$ with $\#A_i \geq \mu(M_1, \ldots , M_t)$ for $i=1, \ldots , m$ which we adjust in such a way that $A_1 \times \cdots \times A_ m\supseteq A$ holds whenever an $A$ is considered. Let $I= \langle G_1, \ldots , G_m \rangle$. Consider the bijective evaluation map 
$${\mbox{ev}} : {\mathbb{F}}[X_1, \ldots , X_m]/I  \rightarrow {\mathbb{F}}^{n=a_1 \cdots a_m}$$
given by ${\mbox{ev}}(F+I)=(F(P_1), \ldots , F(P_n))$  where $\{P_1, \ldots , P_n\}=A_1\times \cdots \times A_m$. Consider 
\begin{eqnarray}
C_2&:=&{\mbox{Span}}_{\mathbb{F}}\{{\mbox{ev}}(N+I) \mid N \prec M_1\} \nonumber \\
C_1&:=&{\mbox{Span}}_{\mathbb{F}}\{{\mbox{ev}}(N+I) \mid N \preceq M_t\}. \nonumber
\end{eqnarray}
Clearly, $C_2 \subseteq C_1$, with codimension equal to $t$. For $1 \leq k \leq t$ consider the relative generalized Hamming weight~\cite{luoetal}
\begin{equation}
M_{k}(C_1,C_2):=\min \{ \# {\mbox{Supp}}(D)  \mid D \subseteq C_1 , D \cap C_2 =\{\vec{0}\}, \dim D = k \} \label{eqsnabelat}
\end{equation}
originally introduced for finite fields, but in the present paper applied to the case of any field. Here, ${\mbox{Supp}}(D)$ means the set of entries $$\{ i \mid i \in \{ 1, \ldots , n\}, c_i \neq 0 {\mbox{ for some }} (c_1, \ldots , c_n) \in D\}.$$
Clearly, for general $F_1, \ldots , F_k$ as in Theorem~\ref{theint} the smallest attainable number of non-common roots in $A_1 \times \cdots \times A_m$ equals $M_k(C_1,C_2)$ which as demonstrated in the coming sections equals~(\ref{eqother}). In the coming sections we shall also demonstrate that if $A\subseteq A_1 \times \cdots  \times A_m$ satisfies $\#A <M_{t-k+1}(C_2^\perp, C_1^\perp)$ then there exist $k$ linearly independent polynomials as in Theorem~\ref{theint} which all evaluates to zero on $A$, but that for some $A\subseteq A_1 \times \cdots \times A_m$ with $\#A=M_{t-k+1}(C_2^\perp, C_1^\perp)$ this does not hold. Here, the dual space (or null-space) is with respect to the usual inner product. Finally, as we shall demonstrate in the next two sections of the paper the right-hand side of~(\ref{eqsnabeleinz}), (\ref{eqsnabeleinzzwei}) serves as a lower bound on $M_{t-k+1}(C_2^\perp, C_1^\perp)$. We conclude that problem (I) and problem (II)  and our answers to them are basically identical, but applied to dual spaces. All claims of the present remark follow by inspection of the proofs of Theorem~\ref{theint} and Theorem~\ref{theintfinite} in Section~\ref{seccorollary1}.
\end{remark}

We conclude this section by demonstrating the superiority of the footprint bound over other known bounds on affine roots counted without multiplicity of general single mulitvariate polynomials. Among these the generalized Alon-Füredi bound~\cite{bishnoi2018zeros} is the sharpest as it has as corollary the Alon-Füredi bound~\cite{alon1993covering} which in turn has as corollaries \cite{bishnoi2018zeros} Alon's combinatorial Nullstellensatz~\cite{alon1999c} as well as the Schwartz-Zippel bound~\cite{1979probabilistic,schwartz1980fast}. For common affine roots of sets of more than one general multivariate polynomial we are not aware of other efficient bounds than the footprint bound~(Theorem~\ref{thefoot}, Corollary~\ref{corfoot} and Corollary~\ref{corsamesame}). 

By inspection the generalized Alon-Füredi bound \cite[Thm.\ 1.2]{bishnoi2018zeros} can be reformulated as follows.

\begin{theorem}\label{thegenalonfuredi}
Consider finite sets $A_1, \ldots , A_m \subseteq {\mathbb{F}}$ and a polynomial $F\in {\mathbb{F}}[X_1, \ldots , X_m]$ with $d_i:=\deg_{X_i}F <a_i$, $i=1, \ldots , m$. Let $d$ be the total degree of $F$. The number of non-roots of $F$ in $A_1 \times \cdots \times A_m$ is at least
\begin{eqnarray}
&&\min \{ \sigma(M) \mid M {\mbox{ is a monomial}}, \nonumber \\
&& {\mbox{ \ \ \ }}\deg M  = d, {\mbox{ and }} \deg_{X_i} M \leq d_i, i=1, \ldots , m\}. \label{eqthegen}
\end{eqnarray}
\end{theorem}

In~\cite{bishnoiblog} it is already noted that the generalized Alon-Füredi bound is a consequence of the footprint bound. We here give some more details including discussing when the two bounds may or may not coincide. Consider a graded monomial ordering, i.e. a monomial ordering $\prec$ satisfying that if $\deg N_1 < \deg N_2$ then necessarily $N_1 \prec N_2$. There are many such orderings, including $m !$ degree lexicographic orderings where in addition to the above property ties are broken by the rule that if $K_1$ and $K_2$ are of the same degree then $K_1 \prec K_2$ if $K_1 \prec_{lex} K_2$ for the given chosen lexicographic ordering. For a graded monomial ordering for sure the leading monomial ${\mbox{lm}}(F)$ belongs to the set considered in~(\ref{eqthegen})
and Theorem~\ref{thegenalonfuredi} follows immediately from Corollary~\ref{corsamesame}. 

We now comment on the possible tightness of the generalized Alon-Füredi bound. If $d_1+\cdots +d_m=d$ then there is only one monomial in the support ${\mbox{Supp}} (F)$ of highest degree, hence the two bounds coincide. Moreover, $F$ is a monomial ordering invariant polynomial~\cite[Def.\ 8]{geil2021multivariate}, i.e.\  all monomials in the support of $F$ divides the unique monomial of highest degree. As shown in~\cite[Cor.\ 9]{geil2021multivariate} the only polynomials that attain the footprint bound in such a case are
 the products of linear terms as described in the proof of Corollary~\ref{corsamesame} (see also~\cite[Sec.\ 3.3]{bishnoi2018zeros}). Imagine next that $d_1+\cdots +d_m\neq d$, in which case there need to be more than one monomial in the argument of~(\ref{eqthegen}). If for some $i$ no monomial $M$ in the support of $F$ satisfies simultaneously $\deg_{X_i} M = d_i$ and $\deg M=d$ then for a  lexicographic ordering with $X_j \prec_{lex}X_i$ for all $j \neq i$ the footprint bound will produce a strictly  tighter result than Theorem~\ref{thegenalonfuredi}.  Example~\ref{exhermite} below illustrates the situation. Assume on the other hand that for all $i$ such a monomial exists, then a nesecary condition for Theorem~\ref{thegenalonfuredi} to produce as sharp a result as the footprint bound is that for each of the graded monomial orderings the largest monomial in the argument of~(\ref{eqthegen}) is of $\sigma$-value equal to~(\ref{eqthegen}) which indeed is a very restrictive assumption on the form of $F$. 

\begin{example}\label{exhermite}
Consider the Hermitian polynomial $F(X_1,X_2)=X_1^{q+1}-X_2^q-X_2 \in {\mathbb{F}}_{q^2}[X_1, X_2]$ where $q$ is a prime-power. We have $d_1=q+1$, $d_2=q$ and $d=q+1$, hence the argument of~(\ref{eqthegen}) is $\{ X_1X_2^q, X_1^{2}X_2^{q-1}, \ldots , X_1^{q+1}\}$ and Theorem~\ref{thegenalonfuredi} therefore tells us that the number of roots of $F$ in ${\mathbb{F}}_{q^2} \times {\mathbb{F}}_{q^2}$ is at most $(q+1)q^2=q^3+q^2$. Applying instead the footprint bound with the monomial ordering being the lexicographic ordering with $X_1 \prec_{lex} X_2$ we obtain ${\mbox{lm}}(F)=X_2^q$ from which we see that $F$ can at most have $q^3$ roots. It is well-known that actually $F$ possesses precisely this amount of affine roots.
\end{example}

\section{Results in linear algebra setting}\label{secclean}
As a preperation for providing in Section~\ref{seccorollary1} a proof of Theorem~\ref{theint} as well as giving the missing details of Remark~\ref{remconnect}, in the current section we treat a problem similar to (II), but in the much more general setting of linear codes over arbitrary fields rather than in the context of multivariate polynomials only. To support future applications beyond multivariate polynomials and for the sake of completeness we also include a treatment of a problem similar to (I) in the setting of linear codes. The main result of the section is Theorem~\ref{thefull} which is the linear code version of Theorem~\ref{theint} and Theorem~\ref{theintfinite}. 

Algebraic coding theory has been a very active research area for several decades leading to a great number of rich constructions of codes having desirable properties. At the heart of the work is the question of estimating the number of roots of polynomial-like functions in various ${\mathbb{F}}_q$-algebras. One attempt to unify some of the theory is the Feng-Rao theory for the ${\mathbb{F}}_q$-algebra $({\mathbb{F}}_q^n, +, \ast)$ where $\ast$ is the component-wise (or element-wise) product given by 
$$(c_1, \ldots , c_n) \ast (d_1, \ldots , d_n)=(c_1d_1, \ldots , c_nd_n).$$
In this section we start by enhancing such results to hold over any field ${\mathbb{F}}$, i.e.\ we demonstrate that the proof of the Feng-Rao bounds holds regardless of underlying field, our exposition closely following that of~\cite{AG, geithom}. With such theory in hand and by applying Forney's two duality lemmas~\cite{forney94} our results concerning multivariate polynomials will follow as easy corollaries.

Consider a subspace $D \subseteq {\mathbb{F}}^n$ and define the support as
$${\mbox{Supp}}(D):=\{i \in {\mathcal{I}}  \mid c_i \neq 0 {\mbox{ for some }}(c_1, \ldots , c_n) \in D\}$$
(we already used this notation in Remark~\ref{remconnect}). 
Note, that this is a different support function than the one we have employed for polynomials earlier in the paper. 
We are interested in establishing information on the size of ${\mbox{Supp}}(D)$ which we call $w_H(D) :=\# {\mbox{Supp}}(D)$. Here, the subscript $H$ refers to the Hamming weight from coding theory. 

Let $B=\{ {\vec{b}}_1, \ldots ,  {\vec{b}}_n\}$ be an ordered basis
for ${\mathbb{F}}^n$. Define
$L_0:=\{ \vec{0} \}$ and for $i=1, \ldots , n$ $L_i:={\mbox{Span}}_{\mathbb{F}}\{ 
\vec{b}_1, \ldots ,\vec{b}_{i}\}$. Further define $\bar{\rho}: {\mathbb{F}}^n \rightarrow \{0, 1, \ldots ,n\}$ by $\bar{\rho}(\vec{0}):=0$ and for $\vec{c} \neq \vec{0}$ 
$\bar{\rho}(\vec{c}):=i$ if $\vec{c} \in L_i \backslash L_{i-1}$ and $m(\vec{c}) :=\min \{ m \mid \vec{c} \cdot \vec{b}_m \neq 0\}$. Here, $\cdot$ is the usual inner product. 
\begin{lemma}\label{lemold}
Consider a subspace $D \subseteq {\mathbb{F}}^n$ of dimension $k\geq 1$. There exists a basis $\{\vec{c}_1, \ldots , \vec{c}_k\}$ for $D$ with $\bar{\rho}(\vec{c}_1), \ldots , \bar{\rho}(\vec{c}_k)$ being pairwise different. For such a basis it holds that $m(\vec{c}_1), \ldots , m(\vec{c}_k)$ are also pairwise different. Both the $\bar{\rho}$-values and the $m$-values are invariants of $D$. Hence, if  $\vec{c} \in D \backslash \{ \vec{0} \} $ then $\bar{\rho}(\vec{c})$ and $m(\vec{c})$, respectively, belongs to the above mentioned sets, respectively. 
\end{lemma}
By Lemma~\ref{lemold} the following definition is clear
$$\bar{\rho}(D):=\{ e \mid  1 \leq e \leq n, \exists \vec{c} \in D \backslash \{\vec{0}\} {\mbox{ with }} \vec{c} \in L_{e} \backslash L_{e-1} \}$$
$$ m(D):=\{m \mid 1 \leq m \leq n, \exists \vec{c} \in  D \backslash \{\vec{0}\} {\mbox{ with }} \vec{c} \in L^\perp_{m-1} \backslash L^\perp_{m} \}$$
and we have $\#  \bar{\rho}(D)=\# m(D) =k$.\\

We next introduce the concept of one-way well-behaving pairs which is more general than the well-known concepts of well-behaving pairs and weakly well-behaving pairs~\cite{matsumotoequiv} allowing for a more powerful description~\cite{geithom} with a wider range of possible future applications of our results.
Consider an additional ordered basis $B^\prime=\{\vec{b}^\prime_1,\ldots \vec{b}^{\prime}_n\}$  (which may or may not be equal to $B$).  Let ${\mathcal{I}}:=\{1,
2, \ldots , n\}$. An ordered pair $(i,j) \in {\mathcal{I}}^2$ is said
to be one-way well-behaving (OWB)
if $\bar{\rho}(\vec{b}_u \ast \vec{b}^{\prime}_j) < \bar{\rho}(\vec{b}_i \ast \vec{b}^{\prime}_j)$
for $u<i$.
Given bases $B, B^{\prime}$ as above consider for $l=1, 2, \ldots ,n$ and $i=1, 2, \ldots ,n$ the following sets
\begin{eqnarray}
{\mbox{V}}_l&:=&\{i \in I \mid  \bar{\rho}(\vec{b}_i \ast \vec{b}_j^{\prime})=l {\mbox{ for some }}
\vec{b}_j^{\prime} \in B^{\prime}   {\mbox{ with  }}  (i,j) {\mbox{ OWB}}\} \\
\Lambda_i&:=&\{ l \in I \mid \bar{\rho}(\vec{b}_{i}\ast \vec{b}_j^{\prime}) = l {\mbox{ for
    some }}   \vec{b}_j^{\prime} \in B^{\prime} {\mbox{ with }}  (i,j)  {\mbox{ OWB}}\}
\end{eqnarray}
\begin{definition}\label{defmusic}
For $\{l_1, \ldots ,l_k\} \subseteq I$ and $\{i_1, \ldots ,i_k\} \subseteq I$
 define
\begin{eqnarray}
\bar{\mu}(l_1, \ldots l_k)&:=&\# \left(  \left( \cup_{s=1, \ldots ,
  k}{\mbox{V}}_{l_s}\right) \cup \{l_1, \ldots ,l_k \}\right) \label{muhh} \\
\bar{\sigma} ( i_1, \ldots i_k   )&:=&\# \left( \left( \cup_{s=1,
      \ldots , k}  \Lambda_{i_s}\right) \cup \{ i_1,  \ldots ,i_k\}\right)
\end{eqnarray} 
\end{definition}

\begin{theorem}\label{egensaet1}
Let $D \subseteq {\mathbb{F}}^n$. Then it holds that 
\begin{eqnarray}
w_H(D)&\geq &\bar{\sigma}(\bar{\rho}(D))  \label{eqprimary} \\
w_H(D)&\geq & \bar{\mu}(m(D)). \label{eqdual}
\end{eqnarray}
\end{theorem}
\begin{proof}
By inspection the proofs in~\cite{AG} and \cite{geithom} do not use any assumption on the field under consideration. \\
To prove~(\ref{eqprimary}) we therefore repeat the proof of \cite{AG}[Thm.\ 10] which with our notation goes as follows. 
Consider a basis $\{\vec{c}_1, \ldots , \vec{c}_k  \}$ for $D$ as in Lemma~\ref{lemold}. 
Write $e_u:=\bar{\rho}(\vec{c}_u)$ for $u=1, \ldots  k$ and observe that if $(e_u,j)$ is OWB for some $j \in \{1, \ldots , n\}$ and $\bar{\rho}(\vec{b}_{e_u} \ast \vec{b}^\prime_j)=\ell$ then by the very definition of OWB also $\bar{\rho}(\vec{c}_u \ast \vec{b}_j^\prime) =\ell$. Hence,
$$W:=\cup_{u=1}^k \{ \vec{c}_u \ast \vec{b}^\prime_j \mid (a_u, j) {\mbox{ is OWB}}\}$$
contains at least $\# \left( \cup_{u=1}^k \Lambda_{e_u} \right)$ linearly independent vectors. Consequently, 
$$W^\prime:=W \cup \{\vec{c}_u \ast (1, \ldots , 1) \mid u=1, \ldots , k\}$$
contains at least $\bar{\sigma}(e_1, \ldots , e_k)$ linearly independent vectors. From this we conclude 
$$\#{\mbox{Supp}} D \geq \dim {\mbox{ Span}}_{\mathbb{F}} W^\prime =\bar{\sigma} (\bar{\rho}(D)).$$
To prove~(\ref{eqdual}) we repeat the proof of the last part of~\cite{geithom}[Thm.\ 1]. Write 
$\gamma:=\bar{\mu}(m_1:=m(\vec{c}_1), \ldots , m_k:=m(\vec{c}_t))$ and
$$
\begin{array}{l}
\{i_1, \ldots , i_\gamma\}:=\\
\cup_{s=1}^{k} \left( \{i \in {\mathcal{I}} \mid \exists \vec{b}^\prime_j \in {\mathcal{B}}^\prime {\mbox{ with }} \bar{\rho}(\vec{b}_i \ast \vec{b}^\prime_j)=m_s {\mbox{ and }} (i,j) {\mbox{ OWB}} \} \cup \{m_s\} \right) 
\end{array}
$$
where without of loss of generality we assume $i_1 < \cdots <i_\gamma$. For $1 \leq h \leq  \gamma$ consider
$$\vec{r}_h=\sum_{v=1}^h \alpha_v \vec{b}_{i_v}, \, \, \alpha_v \in {\mathbb{F}}, \, \, \alpha_h \neq 0$$
and let $T$ be the vector-space consisting of all possible $\vec{r}_h$ as well as $\vec{0}$. \\
If $i_h \in \{m_1, \ldots , m_k\}$ then by the very definition of the function $m$ we have $\vec{r}_h \ast \vec{c}_h \neq 0$ and consequently $\vec{r}_h \ast \vec{c}_h \neq \vec{0}$. If $i_h \notin \{m_1, \ldots , m_k\}$ then there exists a $j$ and an $m_u$, $u\in \{1, \ldots  , k\}$ such that $\bar{\rho} (\vec{b}_{i_h} \ast \vec{b}^\prime_j)=m_u$ with $(i_h,j)$ OWB. It follows that $(\vec{r}_h  \ast \vec{b}^\prime_j)\cdot \vec{c}_u \neq 0$ from which we deduce that $\vec{r}_h \ast \vec{c}_u \neq \vec{0}$. In conclusion for every non-zero element $\vec{r}_h$ of $T$ there exists a $\vec{c} \in D$ with $\vec{r}_h \ast \vec{c} \neq \vec{0}$. It follows that $\#{\mbox{Supp}}(D) \geq \dim T =\gamma$, and we are through.
\end{proof}

\begin{remark}
For the purpose of the present paper the reader may think of ${\mathcal{B}}$ and ${\mathcal{B}}^\prime$ as the same basis. We note, that our results hold also in the general setting where ${\mathcal{B}}^\prime$ is any set, not necessarily a basis. 
\end{remark}

To address (II) at the level of linear algebra it is enough to combine~(\ref{eqdual}) with Forney's duality-lemmas from~\cite{forney94} which we now demonstrate. The concepts of puncturing and projection, respectively, are given as follows. Given a subspace $C \subseteq {\mathbb{F}}^n$ and $A \subseteq {\mathcal{I}}$ we write
$$C_A:=\{ \vec{c}= (c_1, \ldots , c_n) \in C \mid c_i=0 {\mbox{ for all }}i \in {\mathcal{I}} \backslash A=:\bar{A}\}.$$
For $\vec{c}=(c_1, \ldots , c_n)$ we define ${\mathcal{P}}_A(\vec{c}):=(d_1, \ldots , d_n)$ where $d_i:=c_i$ whenever $i \in A$ and $d_i:=0$ otherwise from which we define
$${\mathcal{P}}_A(C):= \{ {\mathcal{P}}_A(\vec{c}) \mid  \vec{c} \in C\}.$$
Forney's duality lemmas \cite{forney94}[Lem.\ 1 and Lem.\ 2] read
\begin{eqnarray}
\dim C&=& \dim C_{\bar{A}}+\dim {\mathcal{P}}_A (C) \label{forney1} \\
\# A&=&\dim {\mathcal{P}}_A(C)+\dim (C^\perp)_A \label{forney2}
\end{eqnarray}
To ease the exposition we now recall the concept of relative generalized Hamming weights, but here described in the setting of arbitrary field (we already applied such notion in~(\ref{eqsnabelat})). Given $C_2 \subsetneq C_1 \subseteq {\mathbb{F}}^n$ write $k_2=\dim C_2$, $k_1=\dim C_1$ and $t=k_1-k_2$. For $1\leq k \leq t$ we define 
$$M_{k}(C_1,C_2):=\min \{ \# {\mbox{Supp}}(D)  \mid D \subseteq C_1 , D \cap C_2 =\{\vec{0}\}, \dim D = k \}$$
and similarly for $C_1^\perp \subsetneq C_2^\perp$. \\

Observe, that in the particular case
\begin{eqnarray}
C_2&=&{\mbox{Span}}_{\mathbb{F}} \{ \vec{b}_1, \ldots , \vec{b}_{k_2}\} \nonumber \\
C_1&=&{\mbox{Span}}_{\mathbb{F}} \{ \vec{b}_1, \ldots , \vec{b}_{k_1}\} \nonumber
\end{eqnarray}
by Theorem~\ref{egensaet1} we obtain
\begin{eqnarray}
M_{k}(C_1,C_2)&\geq &\min \{ \bar{\sigma}( e_1, \ldots , e_k) \mid k_2 < e_1 < \cdots <e_k \leq k_1 \} \label{eqbndfirst} \\
M_{k}(C_2^\perp,C_1^\perp)&\geq &\min \{ \bar{\mu}( e_1, \ldots , e_k) \mid k_2 < e_1 < \cdots <e_k \leq k_1 \}. \label{eqbndsecond}
\end{eqnarray}
We are now ready to state our answer to (II) in the setting of linear algebra, based on which in the next section we shall deduce Theorem~\ref{theint}. 
\begin{theorem}\label{thefull}
Let $1 \leq k_2 < k_1 \leq n$ and consider $C_2={\mbox{Span}}_{\mathbb{F}}\{ \vec{b}_1, \ldots , \vec{b}_{k_2}\} \subseteq C_1={\mbox{Span}}_{\mathbb{F}}\{\vec{b}_1, \ldots , \vec{b}_{k_1} \}$. Write $t=k_1-k_2$ and consider $1 \leq k \leq t$. For all $A \subseteq {\mathcal{I}}$ with 
\begin{equation}
\#A < \min \{ \bar{\mu}(i_1, \ldots , i_{t-k+1}) \mid k_2 < i_1 < \cdots <i_{t-k+1} \leq k_1\}   \label{eqstrictly}
\end{equation}
there exist $k$ linearly independent vectors in $C_1 \backslash C_2$ that are identically equal to $0$ on $A$.\\
When $M_{t-k+1}(C_2^\perp, C_1^\perp)$ equals the right-hand side of~(\ref{eqstrictly}) for some $A$ of size equal to this value there does not exist $k$ such linearly independent vectors.  
\end{theorem}
\begin{proof}
We prove that for all $A \subseteq {\mathcal{I}}$ with $\# A <M_{t-k+1}(C_2^\perp, C_1^\perp)$ the result holds and that for some $A$ with $\# A =M_{t-k+1}(C_2^\perp, C_1^\perp)$ it does not. Combining this with~(\ref{eqbndsecond}) (where we substitute $k$ with $t-k+1$) finalizes the proof. Given $A$ one can write the number of linearly independent vectors in $C_1 \backslash C_2$ being identically equal to $0$ on $A$ as
\begin{eqnarray}
&&\dim (C_1)_{\bar{A}}-\dim (C_2)_{\bar{A }} \nonumber \\
&=&(k_1-\dim P_A(C_1))-(k_2-\dim P_A(C_2)) \nonumber \\
&=& t -\left( \dim {{P}}_A(C_1)-\dim P_A(C_2) \right) \nonumber \\
&=&t - \left( (\# A-\dim (C_1^\perp)_A)-(\# A-\dim (C_2^\perp)_A) \right)  \nonumber \\
&=&t- \left( \dim(C_2^\perp)_A-\dim(C_1^\perp)_A  \right) \nonumber \\
&=&t - \max \{ \dim D \mid D \subseteq C_2^\perp, D\cap C_1^\perp =\{\vec{0}\}, {\mbox{Supp}}(D) \subseteq A\} \nonumber
\end{eqnarray}
where we used Forney's two duality lemmas~(\ref{forney1}) and (\ref{forney2}).
By the very definition of relative generalized Hamming weights for any $A \subseteq {\mathcal{I}}$ of size strictly less than $M_{t-k+1}(C_2^\perp, C_1^\perp)$ we have that 
$$\max \{ \dim D \mid D\subseteq C_2^\perp, D \cap C_1^\perp=\{ \vec{0} \}, {\mbox{Supp}}(D)  \subseteq A\}$$
is strictly smaller than  $t-k+1$ and that for some $A \subseteq {\mathcal{I}}$ with $\# A = M_{t-k+1}(C_2^\perp, C_1^\perp)$ equality holds. This concludes the proof.
\end{proof}

\section{Proof of Theorem~\ref{theint}, Theorem~\ref{theintfinite}, and Corollary~\ref{cornewthing}}\label{seccorollary1}
With the theory from Section~\ref{secclean} in place we are now ready to prove Theorem~\ref{theint} and Theorem~\ref{theintfinite}.\\

\noindent {\textit{Proof of Theorem~\ref{theint}}}\\
Let $q:=\mu(M_1, \ldots , M_t)$ and choose finite sets $A_1, \ldots , A_m \subseteq {\mathbb{F}}$ all of size at least $q$ in such a way that the $A$ in Theorem~\ref{theint} satisfies $A \subseteq A_1 \times \cdots \times A_m$. Write $a_i=\#A_i$, for $i=1, \ldots , m$, $n=a_1 \cdots a_m$ and $\{P_1, \ldots , P_n\}=A_1 \times \cdots \times A_m$. Define $G_i:=\prod_{b\in A_i} (X_i-b)$, for $i=1, \ldots , m$ and consider the ideal $I=\langle G_1, \ldots , G_m\rangle$. The footprint is
$$\Delta_\prec(I)=\{ X_1^{i_1} \cdots X_m^{i_m} \mid 0 \leq i_j < a_j, {\mbox{ for }} j=1, \ldots , m\}=\Box(a_1, \ldots ,a_m)$$
and by interpolation the map ${\mbox{ev}}: {\mathbb{F}}[X_1, \ldots , X_m]/I \rightarrow {\mathbb{F}}^n$ given by ${\mbox{ev}}(F+I)=(F(P_1), \ldots , F(P_n))$ is surjective. Preimage and image are of the same size, and therefore $\{{\mbox{ev}}(N+I) \mid N \in \Delta_\prec(I) \}$ is a basis for ${\mathbb{F}}^n$ as a vector space over ${\mathbb{F}}$. We enumerate the monomials of the footprint according to $\prec$ as $N_1 \prec \cdots \prec N_n$ and write  $\vec{b}_i={\mbox{ev}}(N_i+I)$. 
From the identity ${\mbox{ev}}(F+I) \ast {\mbox{ev}}(G+I)={\mbox{ev}}(FG+I)$ we immediately see that $\mu(N_{i_1}, \ldots , N_{i_k})\leq \bar{\mu}(i_1, \ldots ,i_k)$. But, by the choice of $q$ we have $M_1 =N_i$ for some $i\in \{1, \ldots , n\}$ and $M_s=N_{i-1+s}$ for $s=2, \ldots , t$. Therefore Theorem~\ref{theint} follows from the first part of Theorem~\ref{thefull}. 
\begin{flushright}{$\Box$} \end{flushright}

\noindent {\textit{Proof of Theorem~\ref{theintfinite}}}\\
 The proof ot the first part of Theorem~\ref{theintfinite} is similar to that of Theorem~\ref{theint} except we are already given $A_1 \times \cdots \times A_m$. In the remaining part of the proof we assume that $A_1 \times \cdots \times A_m={\mathbb{F}}_q^m$. Using notation as in~\cite{bras2008duality} consider $W \subseteq \Delta_\prec (I)$ and $E(W):= {\mbox{Span}}_{\mathbb{F}_q} \{ {\mbox{ev}}(M+I) \mid M \in W\}$ and $C(W):= E(W)^\perp$. By~\cite[Prop.\ 2.4]{bras2008duality} whenever $W$ is divisor closed it holds that $C(W)=E(W^\perp)$ where 
\begin{equation}
W^\perp=\Delta_\prec(I) \backslash \{X_1^{q-1-e_1} \cdots X_m^{q-1-e_m} \mid X_1^{e_1} \cdots X_m^{e_m} \in W\}.\label{eqW}
\end{equation}
Here, by divisor closed we mean that if $M \in W$ and $N \mid M$ then also $N \in W$. 
Using the notation as in Theorem~\ref{theintfinite} in combination with that of Theorem~\ref{thefull} define
\begin{eqnarray}
C_2&:=&{\mbox{ev}} \left( ({\mathcal{L}}_2:={\mbox{Span}}_{\mathbb{F}_q} \{ M \mid M \in \Box(q, \ldots , q), M \prec M_1\})+I \right) \nonumber \\
C_1&:=&{\mbox{ev}} \left( ({\mathcal{L}}_1:={\mbox{Span}}_{\mathbb{F}_q} \{ M \mid M \in \Box(q, \ldots , q), M \preceq M_t\})+I \right) \nonumber 
\end{eqnarray}
where of course ${\mbox{lm}}({\mathcal{L}}_1 \backslash {\mathcal{L}}_2)=\{ M_1, \ldots , M_t\}$. From~(\ref{eqW}) we obtain
\begin{eqnarray}
C_2^\perp&=&{\mbox{ev}} \left( ({\mathcal{L}}_2^\perp:={\mbox{Span}}_{\mathbb{F}_q} \{ M \mid M \in \Box(q, \ldots , q), M \preceq \dfrac{K}{M_1}\})+I \right) \nonumber \\
C_1^\perp&=&{\mbox{ev}} \left( ({\mathcal{L}}_1^\perp:={\mbox{Span}}_{\mathbb{F}_q} \{ M \mid M \in \Box(q, \ldots , q), M \prec \dfrac{K}{M_t}\})+I \right) \nonumber
\end{eqnarray}
where $K:=X_1^{q-1}  \cdots X_m^{q-1}$.  Clearly ${\mbox{lm}}({\mathcal{L}}^\perp_2 \backslash {\mathcal{L}}_1^\perp)=\left\{ \dfrac{K}{M_t}, \ldots , \dfrac{K}{M_1}\right\}$ and the last part of Theorem~\ref{theintfinite} follows from the last part of Theorem~\ref{thefull} in combination with the last part of Corollary~\ref{corsamesame}. 
\begin{flushright}{$\Box$} \end{flushright}

\noindent {\textit{Proof of Corollary~\ref{cornewthing}}}\\
Letting $\prec$ be a graded lexicographic ordering and writing $M_1 \prec \cdots \prec M_t$  for the monomials of total degree in the interval $[s_2,s_1]$ according to Theorem~\ref{theint} and Theorem~\ref{theintfinite}, respectively, we must show that 
\begin{equation}\label{eqhuske}
\min \{ \mu (M_{i_1}, \ldots , M_{i_{t-k+1}}) \mid 1 \leq i_1 < \cdots i_{t-k+1} \leq t\}
\end{equation}
equals the expression on the right-hand side of~(\ref{eqsnabel52}) or (\ref{eqsnabel53}) subject to the actual value of $k$. As we shall now demonstrate such equalities can be derived from~\cite{geilmartin2017relative} where relative generalized Hamming weights of $q$-ary Reed-Muller codes were investigated.  For a prime power $q$ and $u \in {\mathbb{N}}_0$ and $m \in {\mathbb{N}}$ the $q$-ary Reed-Muller code is defined as follows
\begin{eqnarray}
&&{\mathcal{RM}}_q(u,m) \nonumber \\
&:=&\{ (F(P_1), \ldots , F(P_n)) \mid F \in {\mathbb{F}}_q[X_1, \ldots , X_m], \deg F \leq u\} \nonumber \\
&=&{\mbox{Span}}_{\mathbb{F}_q}\{ (M(P_1), \ldots , M(P_n)) \mid M \in \Box(q, \ldots , q), \deg M \leq u\}
\end{eqnarray}
where ${\mathbb{F}}_q^m=\{P_1, \ldots , P_{n=q^m}\}$ and $\deg F$ denotes the total degree of $F$. For our purpose we have  $m=2$, i.e.\ two variables. 
In the following we consider a finite field ${\mathbb{F}}_q$ which is equal to the field in Corollary~\ref{cornewthing} when this is finite and which satisfies $s_1 \leq q-1$ when the field in Corollary~\ref{cornewthing} is infinite. Write $u_2=2(q-1)-s_1-1$ and $u_1=u_2+\tau$. From~\cite[Subsubsec.\ 6.1.1]{geilmartin2017relative} we know that for $u_2 \geq q-2$ and $u_2+\tau \leq 2(q-1)$, which holds true under the assumptions of Corollary~\ref{cornewthing},  the codimension of ${\mathcal{RM}}_q(u_2,2) \subseteq {\mathcal{RM}}_q(u_1,2)$ which by definition equals $\dim {\mathcal{RM}}_q(u_1,2) - \dim {\mathcal{RM}}_q(u_2,2)$ equals 
$$\tau (2q-u_2-\tau-2)+\frac{\tau(\tau+1)}{2}=\tau s_2 + \frac{\tau (\tau+1)}{2}=t.$$
Clearly, $t$ then equals the number of monomials in $\Box(q,q)$ of total degree in the interval $[u_2+1,u_2+\tau]$ which we write $N_1 \prec \cdots \prec N_t$. Furthermore, for $\kappa \in \{1, \ldots , \tau (2q-u_2- \tau-2)\}$ write $\kappa=a  \tau +b$ where $a \in \{ 0, \ldots , 2(q-1)-u_2-\tau-1\}$ and $b \in \{1, \ldots , \tau\}$ then 
\begin{eqnarray}
M_\kappa ({\mathcal{RM}}_q(u_2+\tau,2), {\mathcal{RM}}_q(u_2,2))=\left(2q-2-u_2-\frac{a}{2}\right)(a+1)+b-\tau. \label{eqnt1}
\end{eqnarray}
For $\kappa =\tau (2q-u_2-\tau-2)+c$ with $c \in \{ 1, \ldots , \frac{\tau(\tau+1)}{2}\}$ we have
\begin{eqnarray}
&&M_\kappa ({\mathcal{RM}}_q(u_2+\tau,2), {\mathcal{RM}}_q(u_2,2)) \nonumber  \\
&& {\mbox{ \ \ \ \ \ \ }}=\frac{1}{2}(2q-2-u_2-\tau)(2q-u_2+\tau-1)+c. \label{eqnt2}
\end{eqnarray}
The crucial observation now is that by inspection of the proofs in~\cite{geilmartin2017relative} these values equal
$$\min \{ \sigma(N_{i_1}, \ldots , N_{i_\kappa}) \mid 1 \leq i_1 < \cdots < i_\kappa \leq t\}$$
where $\sigma$ is defined as in Corollary~\ref{corfoot} (for $A_1 \times A_2 ={\mathbb{F}}_q^2$). But $\sigma(M_{i_1}, \ldots , M_{i_\kappa})=\mu\big( \frac{K}{M_{i_k}}, \ldots , \frac{K}{M_{i_1}})$ with $K=X_1^{q-1}X_2^{q-1}$ and by proper book-keeping we obtain the desired expressions for (\ref{eqhuske}). Note, that this proof does not assume ${\mathbb{F}}$ to be finite. The $q$-ary Reed-Muller codes over a proper finite field ${\mathbb{F}}_q$ were only employed to establish information on the function $\mu$. If indeed ${\mathbb{F}}$ is finite then one could alternatively have combined (\ref{eqnt1}) and (\ref{eqnt2}) with Theorem~\ref{thefull} without the mentioning of $\sigma$ and $\mu$ for a slightly shorter proof of Corollary~\ref{cornewthing}.
\begin{flushright}{$\Box$} \end{flushright}

\begin{remark}
Inspecting the above proofs of Theorem~\ref{theint} and Theorem~\ref{theintfinite} one can fill in the details missing in Remark~\ref{remconnect}.
\end{remark}

\noindent {\textit{Alternative proof of Corollary~\ref{corsamesame}}}\\
Using similar terminology as in the previous proofs one can demonstrate that Corollary~\ref{corsamesame} is a consequence of (\ref{eqbndfirst}). We leave the details for the reader.
\begin{flushright}{$\Box$} \end{flushright}

\section{Concluding remarks}\label{secconclusion}
In this paper we proposed a unification of common-root and interpolation bounds for any set of affine points and any set of polynomials employing information on the leading monomials, except we did not treat multiplicity. For multivariate polynomials there are several concepts of multiplicity and for more of those a footprint bound has been formulated~\cite[Ch.\ 4, Cor.\ 2.5]{clo2} and \cite{pellikaan2004listextended,geil2019bounding} as well as other types of bounds~\cite{dvir2013extensions,geil2017more}. To the best of our knowledge not much has been done regarding the interpolation problem (II) taking multiplicity into account, see however, ~\cite{walsh2020polynomial}. We leave it for further research to establish such results along the line of Theorem~\ref{theint} and Theorem~\ref{theintfinite} and to link them to the above mentioned solutions of (I) in a fashion similar to Remark~\ref{remconnect}.

We further pose it as a research problem to enhance our findings to hold for integral domains following up on~\cite{bishnoi2018zeros,rote2023generalized} and possibly also to treat skew fields and other non-commutative algebraic structures where the OWB point of view in Section~\ref{secclean} may prove essential.
Similarly it would be interesting to see if some of the methods in the present paper could be applied to treat projective points beyond what is already in the literature e.g.\ in~\cite{fulton2013intersection,mckean2021arithmetic, datta2017number,beelen2018maximum,beelen2019vanishing,beelen2022combinatorial,san2024recursive,nardi2025maximum,singhal2025}.

We note that like other recent results on roots of sets of multivariate polynomials, e.g.~\cite{eisenbud2000projective}, also the present paper is related to work in coding theory and cryptography. For this particular paper the inspiration came from work on ramp secret sharing schemes~\cite{geil2026considerate}. Elaborating further on this relationship we observe, that Theorem~\ref{thefull} should have interesting implications for one-point algebraic geometric codes~\cite[Sec.\ II]{stichtenoth} and other code constructions where one can obtain information on the functions $\bar{\sigma}$ as well as $\bar{\mu}$.\\

The author sincerely thanks Anurag Bishnoi for a fruitful email correspondence during the making of this paper.


\end{document}